\documentclass[12pt]{amsart}
\usepackage{epsfig}
\usepackage{amssymb, amsmath}
\usepackage{amsfonts,graphics,epsfig}
\usepackage{mathrsfs}
\usepackage{pb-diagram, pb-xy}
\usepackage[all]{xy}
\usepackage{comment}
\usepackage{tikz}
\usepackage{setspace}
\usepackage[colorlinks=true, linkcolor=red, urlcolor=blue, citecolor=blue]{hyperref}

\newcommand{\mbR}{\mathbb{R}}

\newcommand{\mbQ}{\mathbb{Q}}

\def\mbP{\mathbb{P}}

\newcommand{\<}{\leq}
\def\>{\geq}

\def\subset{\subseteq}

\newcommand{\tbf}{\textbf}

\newcommand{\num}{\equiv}
\newcommand{\lin}{\sim}
\newcommand{\bir}{\dashrightarrow}
\def\mcO{\mathcal{O}}

\def\injective{\hookrightarrow}

\newtheorem{theorem}{Theorem}[section]
\newtheorem{lemma}[theorem]{Lemma}
\newtheorem{proposition}[theorem]{Proposition}
\newtheorem{corollary}[theorem]{Corollary}
\newtheorem*{mainthma}{Theorem A} 
\newtheorem*{mainthmb}{Theorem B}
\newtheorem*{mainthmc}{Theorem C}

\theoremstyle{remark}
\newtheorem{remark}[theorem]{Remark}
\theoremstyle{definition}
\newtheorem{definition}[theorem]{Definition}
\theoremstyle{definition}

\numberwithin{equation}{section}

\def\Supp{\operatorname{Supp}}
\def\dim{\operatorname{dim}}

\def\chr{\operatorname{char}}
\def\Ex{\operatorname{Ex}}

\def\Spec{\operatorname{Spec}}

\def\Pic{\operatorname{Pic}}
\def\Alb{\operatorname{Alb}}

\def\max{\operatorname{max}}

\author{Omprokash Das}
\address{Department of Mathematics\\
University of California, Los Angeles\\
520 Portola Plaza\\
Math Sciences Building 6363.}
\email{das@math.ucla.edu}

\author{Joe Waldron}
\address{Mathematics Section\\
FSB EPFL SMA\\
Station 8 - MA Building\\
CH-1015 Lausanne\\
 Switzerland.}
\email{joseph.waldron@epfl.ch}

\date{}

\begin{document}
\title[Abundance for $3$-folds]{On the Abundance Problem for $3$-folds in characteristic $p>5$}
\maketitle

\begin{abstract}In this article we prove two cases of the abundance conjecture for $3$-folds in characteristic $p>5$: $(i)$ $(X, \Delta)$ is klt and $\kappa(X, K_X+\Delta)=1$, and $(ii)$ $(X, \Delta)$ is klt, $K_X+\Delta\equiv 0$ and $X$ is not uniruled.
\end{abstract}

\section{Introduction}


The log minimal model program is now known for $3$-folds in characteristic $p>5$ (see for example \cite{HX15}, \cite{Bir16} and \cite{BW14}).  However, the abundance conjecture is still largely open in positive characteristic.  We prove some results in this direction.  Our first result is under the additional assumption of $\kappa(K_X+B)=1$.  Together with the results of \cite{Wal15} and \cite{Bir16} this completes the proof of abundance whenever $\kappa(K_X+B)>0$.  The proof is based on the ideas of Kawamata \cite[Theorem 7.3]{Kaw85} in characteristic $0$, along with some recent results of Tanaka \cite{Tan15, Tan16a}; Birkar, Chen, Zhang \cite{BCZ15}, and Waldron \cite{Wal15}. 

\begin{mainthma}[Theorem \ref{thm:kappa1}]
	Let $(X, \Delta)$ be a projective klt $3$-fold pair over an algebraically closed field $k$ of characteristic $p>5$, such that $K_X+\Delta$ is a nef $\mbQ$-Cartier divisor with $\kappa(X,K_X+\Delta)=1$. Then $K_X+\Delta$ is semi-ample. 
\end{mainthma}

We also obtain the following results in case of $K_X+\Delta\num 0$, when $X$ is not uniruled:

\begin{mainthmb}[Theorem \ref{thm:not-uniruled}]
Let $(X,\Delta)$ be a projective klt $3$-fold pair over an algebraically closed field $k$ of characteristic $p>5$, such that $K_X+\Delta\num 0$ and $X$ is not uniruled.  Then $K_X+\Delta$ is semi-ample, i.e., $K_X+\Delta\sim_\mbQ 0$.\\
\end{mainthmb}

\begin{mainthmc}[Theorem \ref{thm:uniruled-case}]
	Let $(X, \Delta)$ be a projective klt $3$-fold pair over an algebraically closed field $k$ of characteristic $p>5$. Assume that the following conditions are satisfied:
	\begin{enumerate}
	\item $K_X+\Delta\equiv 0$.
	\item The Albanese dimension of $X$ is not equal to $1$.
	\item If the Albanese dimension of $X$ is $2$ and $\Supp\Delta$ intersects the generic fiber of the Albanese morphism, then further assume that $\chr p>\max\{5,\frac{2}{\delta}\}$, where $\delta>0$ is the minimum non-zero coefficient of $\Delta$.\\
	
	 Then $K_X+\Delta\lin_{\mathbb{Q}}0$.\\
	\end{enumerate}
	\end{mainthmc}

In characteristic zero, it is known by \cite{BDPP13} that a smooth variety is uniruled if and only if $K_X$ is not pseudo-effective. However, in positive characteristic there are many examples of uniruled varieties with $\kappa(K_X)\geq 0$ which arise via purely inseparable covers.  Thus ideally we would like to remove this condition. 

We were informed that similar results were obtained independently around the same time by Zhang in \cite{Zha16}; however our techniques seem to be different from his.  

\subsection*{Acknowledgements} 
The authors would like to thank Christopher Hacon, Sho Ejiri, Zsolt Patakfalvi and Karl Schwede for valuable conversations.  We would also like to thank the organizers of the $2015$ AMS Summer Institute in Algebraic Geometry at Utah, where this work began. Part of this work was done when the first named author was visiting the School of Mathematics, Tata Institute of Fundamental Research (TIFR), Mumbai, as a Visiting Fellow (August $2015$-July $2016$). He would like to thank the institute for their accommodation and hospitality.

\section{Preliminaries}

\subsection{Generic fibre}

\begin{lemma}\cite[Lemma 2.20]{BCZ15}\label{lem:integral-generic-fiber}
	Let $f:X\to Y$ be a dominant morphism of finite type between two integral schemes of finite type over a field $k$ of arbitrary characteristic (not necessarily algebraically closed). Let $\eta$ be the generic point of $Y$, and $X_\eta$ the generic fibre. Then the following statements hold:
	\begin{enumerate}
		\item $X_\eta$ is an integral scheme.
		\item $K(X_\eta)\cong K(X)$.
		\item If $x'$ is a point in $X_\eta$ and $x$ its image in $X$ through the set theoretic inclusion $X_\eta\subset X$, then $\mcO_{X_\eta, x'}\cong\mcO_{X, x}$.\\
		\end{enumerate}
In particular, if $X$ is normal (or regular), then $X_\eta$ normal (or regular).\\
		
	\end{lemma}

\begin{corollary}\label{cor:singularities-generic-fiber}
	Let $f:X\to Y$ be a dominant morphism between two varieties with $X$ normal. Let $\eta$ be the generic point of $Y$ and $X_\eta$ the generic fibre. Further assume that $(X, \Delta)$ is a pair such that $K_X+\Delta$ is $\mbQ$-Cartier. If $(X, \Delta)$ has terminal, canonical, klt, plt, dlt or lc singularities, then the pair $(X_\eta, \Delta|_{X_\eta})$ has terminal, canonical, klt, plt, dlt or lc singularities, respectively.
	\end{corollary}

\begin{proof}
	Let $Q\in X$ be a point of $X$ (not necessarily a closed point) and $X_Q=\Spec \mcO_{X, Q}$, where $\mcO_{X, Q}$ is the local ring at $Q$. Then 
	\[\mbox{discrep}(X, \Delta)=\inf\{\mbox{discrep}(X_Q, \Delta_Q): Q\in X\}, \]
	where $\Delta_Q$ is the flat pull back of $\Delta$ by $X_Q\to X$. See \cite[Chapter 2, 2.16]{Kol13} for further discussions.
	
Thus from Lemma \ref{lem:integral-generic-fiber} it follows that if $(X, \Delta)$ has terminal, canonical, klt, plt, dlt or lc singularities, then so does the pair $(X_\eta, \Delta|_{X_\eta})$. 
	
	\end{proof}

\begin{definition}
	Let $f:Y\to X$ be a proper surjective morphism between two normal varieties. A $\mbQ$-divisor $D$ on $Y$ is called \emph{horizontal} with respect to $f$ or \emph{horizontal} over $X$ if $f(\Supp D)=X$, and \emph{vertical} if $f(\Supp D)\neq X$. 
	\end{definition}

\subsection{Iitaka Fibration}\label{subc:iitaka-fibration}
Since the resolution of singularities exists for $3$-folds in characteristic $p>0$ (see \cite{Cut04} and \cite{CP08, CP09}), the Iitaka fibration also exists (see the proof of \cite[Theorem 2.1.33]{Laz04a}).

Let $(X, \Delta)$ be a projective terminal $3$-fold pair in characteristic $p>0$ with $\kappa(X, K_X+\Delta)\geq 0$. Then the Iitaka fibration gives a diagram
		\[
		\xymatrixcolsep{3pc}\xymatrixrowsep{3pc}\xymatrix{  &Y\ar[dl]_\mu\ar[dr]^f&\\
		X && Z}
		\]
satisfying the following conditions:
\begin{enumerate}
	\item $Y$ and $Z$ are smooth projective varieties, $\dim Y=3$ and $\dim Z=\kappa(X, K_X+\Delta)$.
	\item $\mu:Y\to X$ is a log resolution of $(X, \Delta)$ and $f$ is a surjective morphism with $f_*\mcO_Y=\mcO_Z$.
	\item $\kappa(Y_{\eta}, (K_Y+\Delta_Y)|_{Y_{\eta}})=0$, where $\eta$ is the generic point of $Z$, $Y_{\eta}$ is the generic fibre of $f$, and $\Delta_Y=\mu^{-1}_*\Delta$ (see \cite[Lemma 2.3]{LP18}).
	\end{enumerate}

\subsection{Albanese morphism and rational curves}\label{subc:albanese-morphism}
Let $X$ be a normal projective variety over an algebraically closed field $k$ of arbitrary characteristic. From Section $9$ of \cite{FGA} we know that the Albanese morphism $\alpha: X\to \Alb(X)$ exists. It is well known that the induced morphism $\alpha^*:\Pic^0(\Alb(X))\to\Pic^0(X)_{\rm{red}}$ of abelian varieties is an isomorphism of group varieties (see \cite[Chapter 5]{Bad01}). In particular, $\alpha$ induces an isomorphism of groups
\begin{equation}\label{eqn:Pic-0-isomorphism}
	\xymatrixcolsep{3pc}\xymatrix{\alpha^*:\Pic^0(\Alb(X))\ar[r]^\cong & \Pic^0(X).}
\end{equation}

\begin{lemma}\label{lem:abelian-adjunction}
	Let $A$ be an abelian variety of dimension $g>0$, and $Z$ be a subvariety. Then for any resolution of singularities $f:Y\to Z$, $\kappa(Y)\>0$.
	\end{lemma}
\begin{proof}
	Let $dx_i\in H^0(A, \Omega^1_A)$ be a $k$-basis, for $i=1,2,\ldots, g$ and $z\in Z$ a smooth point of $Z$ contained in $Z-f(\Ex(f))$. Then if $\dim(Z)=n$, there exist  $dx_{i_j}$ for $ j=1,2,\ldots, n$  and an open subset $z\in U$ of $Z$ contained in $Z_{\mathrm{smooth}}\cap (Z-f(\Ex(f)))$ such that  
	 $\iota^*dx_{i_1}|_U, \iota^*dx_{i_2}|_U,\ldots, \iota^*dx_{i_n}|_U$ forms a free basis of $\Omega^1_Z(U)$ over $\mcO_Z(U)$, where $\iota:Z\injective A$. Since $\iota^*dx_{i_j}$'s are all global sections of $\Omega^1_Z$, it follows that $\iota^*dx_{i_1}\wedge\iota^*dx_{i_2}\wedge\cdots\wedge\iota^*dx_{i_n}\in H^0(Z, \Omega^n_Z)$ is a non-zero global section of $\Omega^n_Z$. Then $(f\circ\iota)^*(dx_{i_1}\wedge dx_{i_2}\wedge\cdots\wedge dx_{i_n})$ is a non-zero global section of $\Omega^n_Y=\omega_Y$, since it is non-zero on the open set $f^{-1}U$. Therefore $\kappa(Y)=\kappa(Y, \omega_Y)\>0$. 	
	
	\end{proof}

\begin{remark}\label{rmk:albanese-contracts-rational-curves}
From Lemma \ref{lem:abelian-adjunction} it follows that any proper morphism $f:X\to A$ from a variety $X$ to an abelian variety $A$ contracts all rational curves in $X$. 
\end{remark}

\begin{definition}[Uniruled]
A variety $X$ over a base field $k$ (not necessarily algebraically closed) is \emph{uniruled} if there exists a dominant rational map $f:Y\times \mbP^1\dashrightarrow X$ with $\dim(Y)=\dim(X)-1$.
\end{definition}

A proper variety $X$ over an uncountable algebraically closed field $k$ is uniruled if and only if there is a rational curve through a general point of $X$, and also if and only if there is a rational curve through every (closed) point of $X$ (see \cite[Remark 4.1(4)]{Deb01}).\\

\section{Main Theorems}
\subsection{Kodaira dimension $1$}

 In this subsection we will prove abundance when $\kappa(X, K_X+\Delta)=1$. We use some arguments of \cite[Theorem 7.3]{Kaw85}, with the fibre over a general point of the Iitaka fibration replaced by the generic fibre due to the possibility of badly singular closed fibres.

\begin{theorem}\label{thm:kappa1}
	Let $(X, \Delta)$ be a projective klt $3$-fold pair over an algebraically closed field $k$ of characteristic $p>5$, such that $K_X+\Delta$ is a nef $\mbQ$-Cartier divisor with $\kappa(X,K_X+\Delta)=1$. Then $K_X+\Delta$ is semi-ample. 
\end{theorem}

	\begin{proof}
Since $(X, \Delta)$ has klt singularities, by  \cite[Theorem 1.7]{Bir16} there exists a crepant $\mathbb{Q}$-factorial terminal model for $(X,\Delta)$.   We may replace $X$ with this to assume that $(X, \Delta)$ is terminal with $\mbQ$-factorial singularities.
	
		Let the following diagram be the Iitaka fibration of $K_X+\Delta$ as in Subsection \ref{subc:iitaka-fibration}.
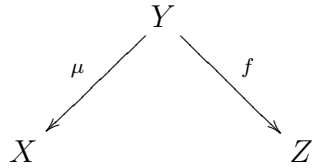
		\[
		\xymatrixcolsep{3pc}\xymatrixrowsep{3pc}\xymatrix{  &Y\ar[dl]_\mu\ar[dr]^f&\\
		X && Z}
		\]

Let $\{E_i\}$ be the exceptional divisors of $\mu$, and  set $\Delta_Y=\mu^{-1}_*\Delta$. Then we have 	
\begin{equation}\label{eqn:mu-log}
	K_Y+\Delta_Y=\mu^*(K_X+\Delta)+\sum r_iE_i, \quad r_i>0 \text{ for all } i.
	\end{equation}	
	
Let $\eta$ be the generic point of $Z$.  Since $(Y, \Delta_Y)$ has klt singularities, $(Y_{\eta}, \Delta_{Y_{\eta}})$ also has klt singularities by Corollary \ref{cor:singularities-generic-fiber}, where $K_{Y_\eta}+\Delta_{Y_\eta}=(K_Y+\Delta_Y)|_{Y_\eta}$.\\

$Y_{\eta}$ is a regular surface over the field $K(Z)$. By \cite[Theorem 1.1]{Tan16a} and \cite[Theorem 1.1]{Tan15} or \cite[Theorem 1.4 and 1.5]{BCZ15}, the LMMP and abundance theorems are known for $(Y_\eta, \Delta_{Y_\eta})$, and thus there exists a projective birational morphism $\sigma:Y_\eta\to W$ such that $K_{W}+\Delta_{W}$ is semi-ample, where $\Delta_{W}=\sigma_*\Delta_{Y_\eta}$. Then we have
\begin{equation}\label{eqn:sigma-log} 
		K_{Y_\eta}+\Delta_{Y_\eta}=\sigma^*(K_{W}+\Delta_{W})+\sum s_jF_j, \quad s_j>0 \text{ for all } i,
	\end{equation}	
	where $\{F_j\}$ are the exceptional divisors of $\sigma$; $s_j>0$ follows from \cite[Lemma 3.38]{KM98}.\\
	
Since $\kappa(W, K_{W}+\Delta_{W})=\kappa(Y_\eta, K_{Y_\eta}+\Delta_{Y_\eta})=0$ and $K_{W}+\Delta_{W}$ is semi-ample, $K_{W}+\Delta_{W}\sim_\mbQ 0$.		
Therefore we have 
\begin{equation}\label{eqn:sigma-log-trivial}
	K_{Y_\eta}+\Delta_{Y_\eta}\sim_\mbQ\sum s_jF_j, \quad s_j>0 \text{ for all } i.
	\end{equation}		
From adjunction on $Y_\eta$ we also get that
\begin{equation}\label{eqn:mu-log-adj}
	K_{Y_\eta}+\Delta_{Y_\eta}\sim_\mbQ\mu^*(K_X+\Delta)|_{Y_\eta}+\sum r_iE_i|_{Y_\eta}
	\end{equation}
From relation \eqref{eqn:sigma-log-trivial} and \eqref{eqn:mu-log-adj} we have,
\begin{equation}\label{eqn:mu-sigma-relation}
	\mu^*(K_X+\Delta)|_{Y_\eta}\sim_\mbQ\sum s_jF_j-\sum r_iE_i|_{Y_\eta}=G^+-G^-,
	\end{equation}		
such that $G^+\>0$ and $G^-\>0$ are two effective $\mbQ$-divisors on $Y_\eta$ with no common irreducible components.\\

We will show that $G^+=G^-=0$. On the contrary first assume that $G^+\neq 0$. It is clear that $G^+$ is $\sigma$-exceptional. Since $Y_\eta$ is a regular excellent surface, by \cite[Theorem 10.1]{Kol13} the intersection matrix of the exceptional divisors of $\sigma: Y_\eta\to W$ is a negative definite matrix. Thus $(G^+)^2<0$; also $G^+\cdot G^-\>0$.\\
On the other hand, since $\mu^*(K_X+\Delta)|_{Y_\eta}$ is nef, from relation \eqref{eqn:mu-sigma-relation} we get $G^+\cdot (G^+-G^-)\>0$, which is a contradiction.	Therefore $G^+=0$ and $\mu^*(K_X+\Delta)|_{Y_\eta}\sim_\mbQ -G^-$. Again since $\mu^*(K_X+\Delta)|_{Y_\eta}$ is nef, $G^-=0$. Therefore we have
\begin{equation}\label{eqn:mu-linearly-0}
	\mu^*(K_X+\Delta)|_{Y_\eta}\sim_\mbQ 0.
	\end{equation}

Now since $K_X+\Delta$ is nef, we may use \cite[Lemma 3.2]{Wal15} on the morphism $f:Y\to Z$ to assume  that there is a $\mathbb{Q}$-Cartier divisor $D$ on $Z$ such that $\mu^*(K_X+\Delta)\lin_{\mathbb{Q}} f^*D$.  Then $D$ is a divisor on a curve with $\kappa(D)=1$, and so is ample.	It then follows that $K_X+\Delta$ is semi-ample.\\ 
\end{proof}

\begin{proposition}\label{prop:relative-abundance}
Let $(X,\Delta)$ be a projective klt $3$-fold pair over an algebraically closed field $k$ of characteristic $p>5$. Let $f:X\to C$ be a projective morphism to a curve $C$ with $f_*\mcO_X=\mcO_C$.  Suppose that $K_X+\Delta$ is $f$-nef.  Then $K_X+\Delta$ is $f$-semi-ample.				
\end{proposition}

\begin{proof}
We follow the proof of \cite[Theorem 1.6]{BCZ15}, where the statement is proved over $\overline{\mathbb{F}}_p$. Let $\eta$ be the generic point of $C$ and $X_\eta$ the generic fibre of $f$. Note that it is commented in \cite{BCZ15} that the assumption of $\overline{\mathbb F}_p$ is used only for the case when $\kappa(K_{X_\eta}+\Delta_{X_\eta})=1$.  So we prove only this case when $k$ is an arbitrary algebraically closed field of characteristic $p>5$.  \cite{BCZ15} also assumes that $\kappa(K_{X_\eta}+\Delta_{X_\eta})\geq 0$, but this is implied by the assumption that $K_X+\Delta$ is $f$-nef together with \cite[Theorem 1.1]{Tan15}. We follow the idea of the proof of \cite[Theorem 1.6]{BCZ15}, with appropriate modifications for the more general field.\\

By assumption, $\kappa(K_{X_\eta}+\Delta_{X_\eta})=1$.  $(X_\eta, \Delta_{X_\eta})$ is a klt surface over the field $K(C)$, and hence by \cite[Theorem 1.5]{BCZ15} or \cite[Theorem 1.1]{Tan15}, $K_{X_\eta}+\Delta_{X_\eta}$ is semi-ample. From this information we construct the following commutative diagram
\begin{equation}
	\xymatrixcolsep{3pc}\xymatrixrowsep{3pc}\xymatrix{& Y\ar[dl]_\phi\ar[dr]^g &\\
	X\ar[rd]_f & & S\ar[dl]^h\\
	& C & }
	\end{equation} 
where $\phi$ is birational and $S$ is a smooth projective surface.   We can see that
\[\phi^*(K_X+\Delta)|_G\sim_\mbQ 0,\]
where $G$ is the generic fibre of $g$, because $K_{X_\eta}+\Delta_{X_\eta}$ is semi-ample.  
Then by \cite[Lemma 3.2]{Wal15} (possibly after replacing $Y$ with a higher model), there exists a $\mbQ$-Cartier divisor $D$ on $S$ such that $\phi^*(K_X+\Delta)\sim_\mbQ g^*D$.

Let $H$ be an ample Cartier divisor on $C$. We claim that $D+mh^*H$ is nef and big on $S$ for all $m\gg 0$.  That it is big is immediate because $D$ is big over $C$.  To see that it is nef, we use
\[\phi^*(K_X+\Delta+mf^*H)\sim_\mbQ g^*(D+mh^*H). \] and show that the left hand side is nef.  This follow from the cone theorem, for if $\Gamma$ is a curve on $X$ such that $(K_X+\Delta)\cdot\Gamma<0$ then $\Gamma$ is not contracted over $C$ because $K_X+\Delta$ is $f$-nef, and so $\Gamma\cdot f^*H$ is a positive integer.  On the other hand, the cone theorem \cite[Theorem 1.7]{Wal17} implies that every $K_X+\Delta$-negative extremal ray contains a curve $\Gamma'$ with $-6\leq (K_X+\Delta)\cdot\Gamma'<0$ and so $(K_X+\Delta+6f^*H)\cdot\Gamma'\geq 0$.

So $K_X+\Delta+mf^*H$ is nef for $m\>6$ and $\kappa(X, K_X+\Delta+mf^*H)=2$. By \cite[Theorem 1]{Tan15m} there exists a $\mbQ$-Cartier divisor $H'\sim_\mbQ mf^*H$ on $X$ such that $(X, \Delta+H')$ is klt. Since $\kappa(X, K_X+\Delta+H')=2$, by \cite[Theorem 1.1]{Wal15}, $K_X+\Delta+H'$ is semi-ample on $X$ and hence $K_X+\Delta$ is semi-ample over $C$.

\end{proof}

\subsection{Numerical dimension $0$ case} In this section we prove some results on abundance when $K_X+\Delta\equiv 0$.

	\begin{theorem}\label{thm:not-uniruled}
Let $(X, \Delta)$ be a projective klt $3$-fold pair over an algebraically closed field $k$ of characteristic $p>5$. Also assume that $K_X+\Delta\num 0$ and $X$ is not uniruled.  Then $K_X+\Delta$ is semi-ample.
\end{theorem}

\begin{remark}\label{rem:boundary-uniruled}
	Note that the boundary divisor $\Delta$ is forced to be $0$ by the assumptions that $X$ is not uniruled and $K_X+\Delta\num 0$, see \cite[Corollary IV.1.14]{Kol96}.
	\end{remark}

\begin{proof}
Since $X$ has klt singularities, a crepant $\mbQ$-factorial terminal model exists by \cite[Theorem 1.7]{Bir16}, 
so we may replace $X$ to assume it has $\mbQ$-factorial terminal singularities.  Note that if $X$ were not already canonical, it would be uniruled by a similar argument as in Remark \ref{rem:boundary-uniruled}.  
 

Let $\alpha:X\to \Alb(X)$ be the Albanese morphism. Since $K_X\equiv 0$,  it follows from the discussion in Subsection \ref{subc:albanese-morphism} and \cite[Theorem 9.6.3]{FGA} that there exists a $\mbQ$-Cartier divisor $L$ on $\Alb(X)$ such that
\begin{equation}\label{eqn:albanese-pullback}
	K_X\sim_\mbQ \alpha^*L.
	\end{equation}
 We split the proof into four different cases depending on the dimension of $\alpha(X)$.\\

\noindent
\tbf{Case I:} $\dim\alpha(X)=0$.\\
In this case $\dim \Alb(X)=\dim\Pic^0(X)=0$.  In this case, $L$ is a divisor on a point, so we have  $K_X\sim_\mbQ 0$.\\

\noindent
\tbf{Case II:} $\dim\alpha(X)=1$.\\
Let $\xymatrixcolsep{2pc}\xymatrix{X\ar[r]^f & C\ar[r] & \alpha(X)}$ be the Stein factorization of $\alpha$. 

Let $\eta$ be the generic point of $C$.  Since $X$ is not uniruled, by \cite[Theorem 1.2]{Tan15f} the generic fiber $X_\eta$ is geometrically normal, i.e., $X_{\bar{\eta}}$ is normal.\\

Let $\pi:Y\to X_{\bar{\eta}}$ be the minimal resolution of the geometric generic fibre $X_{\bar{\eta}}$ of $f:X\to C$.  Then we have
 \[K_Y+\Delta = \pi^*(K_{X_{\bar{\eta}}}),\]
 for some $\mbQ$-divisor $\Delta\>0$.\\
We claim that $\Delta=0$, if not, i.e., if $\Delta>0$, then by running a $K_Y$-MMP we end up with a Mori fiber space, which gives a uniruling of $X_{\bar{\eta}}$, which in turns gives a uniruling of $X_{\eta}$, and thus a uniruling of $X$, a contradiction.  Hence $\Delta = 0$, and in particular, $X_{\bar{\eta}}$ has canonical singularities, which are strongly $F$-regular for surfaces in characteristic $p>5$ by \cite[Corollary 4.9]{Har98}.\\

We will show that $C$ is an elliptic curve.  Since $(X_{\bar{\eta}}, 0)$ is strongly $F$-regular, by \cite[Theorem B and Corollary 4.22]{PSZ13}, the general fibers of $f:X\to C$ are also strongly $F$-regular. Since $K_X$ is a pullback from the base $C$, by \cite[Theorem 3.16]{Pat14} $K_{X/C}\num -f^*K_C$ is nef.  It then follows that $C$ must be an elliptic curve, as it cannot be rational by Lemma \ref{lem:abelian-adjunction} and Remark \ref{rmk:albanese-contracts-rational-curves}. Furthermore, from the universal property of the Albanese morphism it then follows that $C\cong \Alb(X)$.\\

Since $f: X\to C$ is the Albanese morphism, from \eqref{eqn:albanese-pullback} we have $K_X\sim_\mbQ f^*(K_C+L)$, for some $\mbQ$-divisor $L$ on $C$.  Then by \cite[Theorem 3.2]{EZ16}, $L$ is semi-ample, and thus $K_X\lin_{\mathbb{Q}} 0$.\\

\noindent
\tbf{Case III:} $\dim\alpha(X)=2$.\\
Let $\xymatrixcolsep{2pc}\xymatrix{X\ar[r]^f & Y\ar[r]^g & V}$ be the Stein factorization of $\alpha$, where $V=\alpha(X)$. Let $\phi_Y:Y'\to Y$ be the minimal resolution of $Y$, and $\phi_X:X'\to X$ a resolution of the graph of the induced rational map $X\bir Y'$ sitting in the following commutative diagram:
\begin{equation}
	\xymatrixcolsep{3pc}\xymatrixrowsep{3pc}\xymatrix{ X'\ar[d]_{\phi_X}\ar[r]^{f'} & Y'\ar[d]^{\phi_Y}\\
	X\ar[r]_f\ar@{-->}[ur] & Y. }
	\end{equation}
Note that since $X$ has terminal singularities and the rational map $X\bir Y'$ is defined as a morphism over a dense open subset $U'\subset Y'$,  we can choose $X'$ (a resolution of the graph of $X\bir Y'$) in such a way that the exceptional divisors of $\phi_X:X'\to X$ do not intersect the generic fibre of $f':X'\to Y'$. So in particular we still have $K_{F'}=K_{X'}|_{F'}\sim_\mbQ 0$, where $F'$ is the generic fibre of $f'$.\\

Since $X$ is not uniruled, by  \cite[Theorem 1.2]{Tan15f} the geometric generic fibre $F'_{\bar{\eta}}$ of $f'$ is normal, i.e., $F'_{\bar{\eta}}$ is a smooth elliptic curve over $k(\bar{\eta})=\overline{K(Y)}$. Therefore the general fibres of $f'$ are all smooth elliptic curves, and hence by \cite[2.14]{BCZ15}, there is an effective divisor $E'\>0$ on $X'$ such that 
$$K_{X'}\lin_\mathbb{Q}f'^* K_{Y'} +E'.$$
If we show that $\kappa(K_{Y'})\geq 0$, then it follows that $\kappa(K_{X'})\geq0$.  As $X'$ has a birational morphism to $X$, this in turn implies that $\kappa(K_X)\geq 0$ by the projection formula. By construction, $Y'$ is a smooth surface of maximal Albanese dimension, in particular by Remark \ref{rmk:albanese-contracts-rational-curves} an open set of $Y'$ is disjoint from all rational curves in $Y'$. Therefore by the classification of surfaces (\cite[Theorem V.6.1]{Har77}) $\kappa(K_{Y'})\geq 0$.\\

\noindent
\tbf{Case IV:} $\dim\alpha(X)=3$\\
This case follows from Theorem \ref{thm:maximal-albanese-abundance} in the Appendix, which is due to Christopher Hacon.\\
\end{proof}

\begin{theorem}\label{thm:uniruled-case}
	Let $(X, \Delta)$ be a projective klt $3$-fold pair over an algebraically closed field $k$ of characteristic $p>5$. Assume that the following conditions are satisfied
	\begin{enumerate}
	\item $K_X+\Delta\equiv 0$.
	\item The Albanese dimension of $X$ is not equal to $1$.
	\item If the Albanese dimension of $X$ is $2$ and $\Supp\Delta$ intersects the generic fiber of the Albanese morphism, then further assume that $\chr p>\max\{5, \frac{2}{\delta}\}$, where $\delta>0$ is the minimum non-zero coefficient of $\Delta$.\\
	
	 Then $K_X+\Delta\lin_{\mathbb{Q}}0$.
	\end{enumerate}
	\end{theorem}

\begin{proof} 
Since $(X, \Delta)$ has klt singularities, by \cite[Theorem 1.7]{Bir16} we may assume that $(X, \Delta)$ has $\mbQ$-factorial terminal singularities.\\

If the Albanese dimension of $X$ is $0$, then by a similar argument as in the proof of Theorem \ref{thm:not-uniruled} it follows that $K_X+\Delta\sim_\mbQ 0$.\\

If the Albanese dimension of $X$ is $3$, and $\Delta\neq 0$, then from \cite[Corollary IV.1.14]{Kol96} it follows that $X$ is uniruled, which contradicts the assumption that $X$ has maximal Albanese dimension by Remark \ref{rmk:albanese-contracts-rational-curves}. Thus $\Delta=0$, and in this case the proof follows from Hacon's Theorem
 \ref{thm:maximal-albanese-abundance} in the appendix.\\

Now we will deal with the Albanese dimension $2$ case.\\
 Let $\xymatrixcolsep{2pc}\xymatrix{X\ar[r]^f & Y\ar[r]^g & V}$ be the Stein factorization of the Albanese morphism $\alpha: X\to\Alb(X)$, where $V=\alpha(X)$. We have $K_X+\Delta\sim_\mbQ f^*M$ (by \ref{eqn:Pic-0-isomorphism}) for some $\mbQ$-Cartier divisor $M$ on $Y$. There are two cases:\\
 
\noindent 
\textbf{Case I:} The support of $\Delta$ does not intersect the generic fiber of $f:X\to Y$.\\
In this case the proof runs identically to the Case III of Theorem \ref{thm:not-uniruled} with the following modifications.\\
In paragraph three of the Case III of Theorem \ref{thm:not-uniruled} we replace the arguments for $F'_{\bar{\eta}}$ being smooth with the following:  Since $\Supp\Delta$ does not intersect the generic fiber of $f:X\to Y$, by setting $U=Y\setminus f(\Supp\Delta)$ we see that $K_X\num_f 0$ over $U\subset Y$. Then by \cite[Corollary 1.5]{PW17} the geometric generic fiber $F_{\bar{\eta}}$ of $f$ is a smooth elliptic curve over ${k(\bar{\eta})}=\overline{K(Y)}$. In particular, $F'_{\bar{\eta}}$ is a smooth elliptic curve over ${k(\bar{\eta})}=\overline{K(Y')}=\overline{K(Y)}$, since the generic fibers of $f:X\to Y$ and $f':X'\to Y'$ are isomorphic by construction.\\

\noindent
\textbf{Case II:} The support of $\Delta$ intersects the generic fiber of $f:X\to Y$.\\ 
In this case let $\delta>0$ be the minimum non-zero coefficient of $\Delta$. Then the result follows from \cite[Proposition 6.4]{CTX15}.\\

	\end{proof}


\appendix

\section{Abundance for varieties of maximal Albanese dimension and $K_X\equiv 0$}
\begin{center}
\author{Christopher D. Hacon}\\ 
\address{Department of Mathematics \\  
University of Utah\\  
Salt Lake City, UT 84112, USA}\\
\email{hacon@math.utah.edu}
\end{center}

\begin{theorem}\label{thm:maximal-albanese-abundance} Let $X$ be a normal projective variety over an algebraically closed field $k$ of characteristic $p>0$ and $a:X\to A$ the Albanese morphism. If $\dim a(X)=\dim X$, $K_X$ is $\mbQ$-Cartier and $K_X\equiv 0$, then $(p-1)K_X\sim 0$.
\end{theorem}
\begin{proof} We follow the notation and conventions of \cite{HP13}.
Since $a$ is generically finite, it follows that $a_*F_*\omega _X\to a_*\omega _X$ is generically surjective and so $S^0a_*\omega _X\ne 0$ and $\Omega \ne 0$.  By the proof of \cite[4.2.6]{HP13} there exists $P\in {\rm Pic}^0(A)$ such that $0\ne H^0(S^0a_* \omega _X\otimes P)\subset H^0(\omega _X\otimes a^*P)$. It follows that  $K_X\sim  a^*P^\vee$. 
Let $$V^0(S^0a_*\omega _X):=\{Q\in {\rm Pic}^0(A)|  h^0(S^0a_*\omega _X\otimes Q)\ne 0\}.$$ 
Since $\{ P \}\subset V^0(S^0a_*\omega _X)\subset V^0(a_*\omega _X)=\{ P\}$, it follows that $$V^0(S^0a_*\omega _X)=\{ P \}$$ consists of a unique point.
Let $\Lambda _e=R\hat SD_A(F^e_*S^0a_*\omega _X)$ and $\Lambda ={\rm hocolim}(\Lambda _e)\in D^b_{\rm qcoh}(\hat A)$.
By \cite[3.1.2]{HP13}, $\Lambda \cong \varinjlim \mathcal H ^0(\Lambda _e)$ is a quasi-coherent sheaf. Let $\Lambda '_e={\rm im}(\mathcal H ^0(\Lambda _e)\to \Lambda)$, then $\Lambda =\varinjlim \Lambda '_e$. By \cite[3.1.1, 3.1.2]{HP13}, $$\Omega :=\varprojlim F^e_* S^0a_*\omega _X=((-1_A)^*D_ARS(\Lambda ))[-g] .$$ Since $\Omega \ne 0$, it follows that $\Lambda _e'\ne 0$. Let $Z$ be the support of $\Lambda '_0$, then $Z\ne \emptyset$ by what we have just observed. On the other hand, by cohomology and base change, $Z\subset {\rm Supp} \mathcal H ^0(\Lambda _0)\subset \{ P\}$ and so
$Z=\{P\}$. 
By the proof of \cite[3.3.5]{HP13}, we have that $pZ\subset Z$. It follows that $P^{\otimes p}\cong P$ so that $P^{\otimes p-1}\cong \mathcal O _A$. 
\end{proof}

\bibliographystyle{hep}
\bibliography{references.bib}

\newcommand{\etalchar}[1]{$^{#1}$}
\begin{thebibliography}{{Tan}15b}

\bibitem[B{\u{a}}d01]{Bad01}
L.~B{\u{a}}descu,
\newblock \textsl{ Algebraic surfaces},
\newblock Universitext, Springer-Verlag, New York, 2001,
\newblock Translated from the 1981 Romanian original by Vladimir Ma{\c{s}}ek
  and revised by the author.

\bibitem[BCZ16]{BCZ15}
C.~Birkar, Y.~Chen and L.~Zhang, \textsl{ Iitaka's ${C}_{n,m}$ conjecture for
  3-folds over finite fields},
\newblock Nagoya Math. J. , 1--31 (2016).

\bibitem[BDPP13]{BDPP13}
S.~Boucksom, J.-P. Demailly, M.~P{\u{a}}un and T.~Peternell, \textsl{ The
  pseudo-effective cone of a compact {K}\"ahler manifold and varieties of
  negative {K}odaira dimension},
\newblock J. Algebraic Geom. \textbf{ 22}(2), 201--248 (2013).

\bibitem[Bir16]{Bir16}
C.~Birkar, \textsl{ Existence of flips and minimal models for 3-folds in char
  {$p$}},
\newblock Ann. Sci. \'Ec. Norm. Sup\'er. (4) \textbf{ 49}(1), 169--212 (2016).

\bibitem[BW17]{BW14}
C.~Birkar and J.~Waldron, \textsl{ Existence of {Mori} fibre spaces for 3-folds
  in char p},
\newblock Adv. Math. \textbf{ 313}, 62--101 (2017).

\bibitem[CP08]{CP08}
V.~Cossart and O.~Piltant, \textsl{ Resolution of singularities of threefolds
  in positive characteristic. {I}. {R}eduction to local uniformization on
  {A}rtin-{S}chreier and purely inseparable coverings},
\newblock J. Algebra \textbf{ 320}(3), 1051--1082 (2008).

\bibitem[CP09]{CP09}
V.~Cossart and O.~Piltant, \textsl{ Resolution of singularities of threefolds
  in positive characteristic. {II}},
\newblock J. Algebra \textbf{ 321}(7), 1836--1976 (2009).

\bibitem[CTX15]{CTX15}
P.~Cascini, H.~Tanaka and C.~Xu, \textsl{ On base point freeness in positive
  characteristic},
\newblock Ann. Sci. \'Ec. Norm. Sup\'er. (4) \textbf{ 48}(5), 1239--1272
  (2015).

\bibitem[Cut04]{Cut04}
S.~D. Cutkosky,
\newblock \textsl{ Resolution of singularities}, volume~63 of \textsl{ Graduate
  Studies in Mathematics},
\newblock American Mathematical Society, Providence, RI, 2004.

\bibitem[Deb01]{Deb01}
O.~Debarre,
\newblock \textsl{ Higher-dimensional algebraic geometry},
\newblock Universitext, Springer-Verlag, New York, 2001.

\bibitem[EZ16]{EZ16}
S.~{Ejiri} and L.~{Zhang}, \textsl{ {Iitaka's $C_{n,m}$ conjecture for 3-folds
  in positive characteristic}},
\newblock ArXiv e-prints  (April 2016), {1604.01856}.

\bibitem[FGI{\etalchar{+}}05]{FGA}
B.~Fantechi, L.~G{\"o}ttsche, L.~Illusie, S.~L. Kleiman, N.~Nitsure and
  A.~Vistoli,
\newblock \textsl{ Fundamental algebraic geometry}, volume 123 of \textsl{
  Mathematical Surveys and Monographs},
\newblock American Mathematical Society, Providence, RI, 2005,
\newblock Grothendieck's FGA explained.

\bibitem[Har77]{Har77}
R.~Hartshorne,
\newblock \textsl{ Algebraic geometry},
\newblock Springer-Verlag, New York, 1977,
\newblock Graduate Texts in Mathematics, No. 52.

\bibitem[Har98]{Har98}
N.~Hara, \textsl{ Classification of two-dimensional {$F$}-regular and
  {$F$}-pure singularities},
\newblock Adv. Math. \textbf{ 133}(1), 33--53 (1998).

\bibitem[HP16]{HP13}
C.~D. Hacon and Z.~Patakfalvi, \textsl{ Generic vanishing in characteristic
  {$p>0$} and the characterization of ordinary abelian varieties},
\newblock Amer. J. Math. \textbf{ 138}(4), 963--998 (2016).

\bibitem[HX15]{HX15}
C.~D. Hacon and C.~Xu, \textsl{ On the three dimensional minimal model program
  in positive characteristic},
\newblock J. Amer. Math. Soc. \textbf{ 28}(3), 711--744 (2015).

\bibitem[Kaw85]{Kaw85}
Y.~Kawamata, \textsl{ Pluricanonical systems on minimal algebraic varieties.},
\newblock Inventiones mathematicae \textbf{ 79}, 567--588 (1985).

\bibitem[KM98]{KM98}
J.~Koll{\'a}r and S.~Mori,
\newblock \textsl{ Birational geometry of algebraic varieties}, volume 134 of
  \textsl{ Cambridge Tracts in Mathematics},
\newblock Cambridge University Press, Cambridge, 1998,
\newblock With the collaboration of C. H. Clemens and A. Corti, Translated from
  the 1998 Japanese original.

\bibitem[Kol96]{Kol96}
J.~Koll{\'a}r,
\newblock \textsl{ Rational curves on algebraic varieties}, volume~32 of
  \textsl{ Ergebnisse der {Mathematik} und ihrer {Grenzgebiete}. 3. {Folge}.
  {A} {Series} of {Modern} {Surveys} in {Mathematics} [{Results} in
  {Mathematics} and {Related} {Areas}. 3rd {Series}. {A} {Series} of {Modern}
  {Surveys} in {Mathematics}]},
\newblock Springer-Verlag, Berlin, 1996.

\bibitem[Kol13]{Kol13}
J.~Koll{\'a}r,
\newblock \textsl{ Singularities of the minimal model program}, volume 200 of
  \textsl{ Cambridge Tracts in Mathematics},
\newblock Cambridge University Press, Cambridge, 2013,
\newblock With a collaboration of S{\'a}ndor Kov{\'a}cs.

\bibitem[Laz04]{Laz04a}
R.~Lazarsfeld,
\newblock \textsl{ Positivity in algebraic geometry. {I}}, volume~48 of
  \textsl{ Ergebnisse der Mathematik und ihrer Grenzgebiete. 3. Folge. A Series
  of Modern Surveys in Mathematics [Results in Mathematics and Related Areas.
  3rd Series. A Series of Modern Surveys in Mathematics]},
\newblock Springer-Verlag, Berlin, 2004,
\newblock Classical setting: line bundles and linear series.

\bibitem[LP18]{LP18}
V.~Lazi\'c and T.~Peternell, \textsl{ {Abundance for varieties with many
  differential forms}},
\newblock {\'Epijournal de G\'eom\'etrie Alg\'ebrique} \textbf{ {Volume 2}}
  (February 2018).

\bibitem[Pat14]{Pat14}
Z.~Patakfalvi, \textsl{ Semi-positivity in positive characteristics},
\newblock Ann. Sci. \'Ec. Norm. Sup\'er. (4) \textbf{ 47}(5), 991--1025 (2014).

\bibitem[PSZ13]{PSZ13}
Z.~{Patakfalvi}, K.~{Schwede} and W.~{Zhang}, \textsl{ {F-singularities in
  families}},
\newblock ArXiv e-prints  (May 2013), {1305.1646}.

\bibitem[PW17]{PW17}
Z.~Patakfalvi and J.~Waldron, \textsl{ Singularities of general fibers and the
  LMMP},
\newblock ArXiv e-prints  (August 2017), {1708.04268}.

\bibitem[{Tan}15a]{Tan15}
H.~{Tanaka}, \textsl{ {Abundance theorem for surfaces over imperfect fields}},
\newblock ArXiv e-prints  (February 2015), {1502.01383}.

\bibitem[{Tan}15b]{Tan15m}
H.~{Tanaka}, \textsl{ {Semiample perturbations for log canonical varieties over
  an F-finite field containing an infinite perfect field}},
\newblock ArXiv e-prints  (March 2015), {1503.01264}.

\bibitem[{Tan}16]{Tan16a}
H.~{Tanaka}, \textsl{ {Minimal model program for excellent surfaces}},
\newblock ArXiv e-prints  (August 2016), {1608.07676}.

\bibitem[Tan17]{Tan15f}
H.~Tanaka, \textsl{ Semiample perturbations for log canonical varieties over an
  {$F$}-finite field containing an infinite perfect field},
\newblock Internat. J. Math. \textbf{ 28}(5), 1750030, 13 (2017).

\bibitem[Wal17a]{Wal15}
J.~Waldron, \textsl{ Finite generation of the log canonical ring for 3-folds in
  char $p$},
\newblock Math. Res. Lett. \textbf{ 24}(3), 933--946 (2017).

\bibitem[Wal17b]{Wal17}
J.~Waldron, \textsl{ The LMMP for log canonical 3-folds in characteristic
  $p>5$},
\newblock Nagoya Mathematical Journal , 1–--24 (2017).

\bibitem[{Zha}17]{Zha16}
L.~{Zhang}, \textsl{ Abundance of non-uniruled 3-folds with non-trivial
  Albanese maps in positive characteristics},
\newblock ArXiv e-prints  (May 2017), {1610.03637v3}.

\end{thebibliography}

\end{document}